\documentclass[11pt]{amsart}

\usepackage{amsmath}
\usepackage{fullpage}
\usepackage{xspace}
\usepackage[psamsfonts]{amssymb}
\usepackage[latin1]{inputenc}
\usepackage{graphicx,color}
\usepackage[curve]{xypic} 
\usepackage{hyperref}
\usepackage{graphicx}
\usepackage{amsmath}%
\usepackage{amsthm}%
\usepackage{amscd}
\usepackage{amsfonts}%
\usepackage{amssymb}%
\usepackage{graphicx}

\usepackage{mathrsfs}

\usepackage{tikz}
\usetikzlibrary{matrix,arrows}

\usepackage{tikz-cd}

\newtheorem{theorem}{Theorem}[section]

\newtheorem{conjecture}[theorem]{Conjecture}
\newtheorem{question}[theorem]{Question}
\newtheorem{corollary}[theorem]{Corollary}

\theoremstyle{remark}

\numberwithin{equation}{section}

\newcommand{\Lcal}{\mathscr{L}}

\newcommand{\Pro}{\mathbb{P}}

\newcommand{\Q}{\mathbb{Q}}

\newcommand{\G}{\mathbb{G}}

\newcommand{\rk}{\mathrm{rank}\,}

\makeatletter
\@namedef{subjclassname@2020}{%
  \textup{2020} Mathematics Subject Classification}
\makeatother

  \DeclareFontFamily{U}{wncy}{}
    \DeclareFontShape{U}{wncy}{m}{n}{<->wncyr10}{}
    \DeclareSymbolFont{mcy}{U}{wncy}{m}{n}
    \DeclareMathSymbol{\Sha}{\mathord}{mcy}{"58}

\begin{document}
\title[A note on Bremner's conjecture and uniformity]{A note on Bremner's conjecture and uniformity}

\author{Natalia Garcia-Fritz}
\address{ Departamento de Matem\'aticas,
Pontificia Universidad Cat\'olica de Chile.
Facultad de Matem\'aticas,
4860 Av.\ Vicu\~na Mackenna,
Macul, RM, Chile}
\email[N. Garcia-Fritz]{natalia.garcia@uc.cl}%

\author{Hector Pasten}
\address{ Departamento de Matem\'aticas,
Pontificia Universidad Cat\'olica de Chile.
Facultad de Matem\'aticas,
4860 Av.\ Vicu\~na Mackenna,
Macul, RM, Chile}
\email[H. Pasten]{hector.pasten@uc.cl}%

%\thanks{}
\thanks{N.G.-F. was supported by ANID Fondecyt Regular grant 1251300 from Chile. H.P. was supported by ANID Fondecyt Regular grant 1230507 from Chile.}
\date{\today}
%\dedicatory{}

\begin{abstract} In 1998, Bremner conjectured that elliptic curves over the rationals having long sequences of distinct rational points whose $x$-coordinates are in arithmetic progression, have large rank. This was proved some years ago in a strong form as a consequence of previous work by the authors, by a combination of Nevanlinna theory and the uniform Mordell--Lang theorem of Gao--Ge--K\"uhne. Thus, if the ranks of elliptic curves over the rationals are uniformly bounded, then so are the lengths of the aforementioned arithmetic progressions. In this note we give a much more direct proof of this last statement, using the height-uniform Mordell theorem of Dimitrov--Gao--Habegger. The method is flexible and we give a new application of these ideas to $x$-coordinates in finitely generated multiplicative groups and geometric progressions; connections to a possible semiabelian uniform Mordell--Lang are also discussed.
\end{abstract}

\maketitle

%\tableofcontents

%%%%%%%%%%%%%%%%%%%%%%%%%%%%%%%%%%%%%%
%%%%%%%%%%%%%%%%%%%%%%%%%%%%%%%%%%%%%%
%%%%%%%%%%%%%%%%%%%%%%%%%%%%%%%%%%%%%%
%%%%%%%%%%%%%%%%%%%%%%%%%%%%%%%%%%%%%%
%%%%%%%%%%%%%%%%%%%%%%%%%%%%%%%%%%%%%%
%%%%%%%%%%%%%%%%%%%%%%%%%%%%%%%%%%%%%%

\section{Introduction}

\subsection{Bremner's rank conjecture} For an elliptic curve $E$ over $\Q$, an \emph{arithmetic progression of length $M$} is a sequence of points $P_1,...,P_M$ in $E(\Q)$ whose $x$-coordinates with respect to one (equivalently, all) Weierstrass equation $y^2=f(x)$ with $f$ cubic, form a non-trivial arithmetic progression in $\Q$.

In \cite{Bremner} Bremner conjectured that elliptic curves over $\Q$ with long arithmetic progressions have large rank. To be accurate, Bremner says:

\emph{It seems that points of an arithmetic progression have the tendency to be linearly independent in the group of rational points (...)}

As explained in \cite{Bremner}, this conjecture is in part motivated by \cite{BST} which studies arithmetic progressions on rank one elliptic curves in quadratic twist families. There is a large body of work studying arithmetic progressions on elliptic curves ---both experimental and theoretical--- and we refer to \cite{GFP} for a literature review.

Bremner's rank conjecture was already proved as a consequence of our previous work \cite{GFP} via a combination of Nevanlinna theory and an application of the uniform Mordell--Lang theorem of Gao--Ge--K\"uhne \cite{GGK} (see Section \ref{SecProofHistory} for details). 

Our aim here is twofold. We give an alternative and very short proof of a (conditional) uniformity consequence of Bremner's conjecture, and we discuss a new uniformity phenomenon involving multiplicative groups and $x$-coordinates of rational points in elliptic curves. Both topics follow similar ideas, where the main point is an auxiliary curve of genus $2$ with split Jacobian.

%%%%%%
%%%%%%
%%%%%%

\subsection{Bremner's uniformity question} In the same paper \cite{Bremner}, even before discussing ranks, Bremner explicitly asked the following:

\begin{question}[Bremner's uniformity question]\label{QuestionB}
Can there exist arbitrarily large arithmetic progressions on elliptic curves (over $\Q$)?
\end{question}

Regarding uniform boundedness of arithmetic progressions on Mordell elliptic curves $y^2=x^3+k$, this particular family has attracted special attention under the name of Mohanty's conjecture and we refer the reader to \cite{GF} and the references therein for further discussion.

 The following immediate consequence of Theorem \ref{ThmBC} can be regarded as a conditional answer to Question \ref{QuestionB} and, in fact, it was consequences of this kind that served as a key motivation in our work \cite{GFP}.

\begin{theorem}\label{ThmUniform} If the ranks of elliptic curves over $\Q$ are uniformly bounded, then so are the lengths of arithmetic progressions on elliptic curves over $\Q$.
\end{theorem}

We refer the reader to \cite{PPVW} for a detailed study of the question of uniform boundedness of ranks of elliptic curves, especially over $\Q$. 

In Section \ref{SecProof} we present a simple and short proof of Theorem \ref{ThmUniform} that avoids the heavy machinery of Nevanlinna theory that we used in \cite{GFP} and which ``only'' uses the height-uniform Mordell theorem of Dimitrov--Gao--Habegger \cite{DGH}. Key to our argument is the use of genus $2$ curves with split Jacobian.

The result also holds over number fields, but the question of uniform boundedness of ranks in more generality than $\Q$ seems more dubious, so we keep the discussion over $\Q$ (which, by the way, is the original setting discussed by Bremner \cite{Bremner}).

%%%%%%
%%%%%%
%%%%%%

\subsection{Another uniformity result} 

Bremner's uniformity question relates the group structure of an elliptic curve to the additive structure of the affine line. A natural problem is to use the $x$-coordinate map to relate the group structure in the elliptic curve to finitely generated multiplicative groups on $\G_m$. Using the same method of constructing genus $2$ curves with split Jacobian as in our new proof of Theorem \ref{ThmUniform}, we prove a result that hints at a new uniformity question:

\begin{theorem}\label{ThmMult} If the ranks of elliptic curves over $\Q$ are uniformly bounded, then there is a constant $\kappa>1$ with the following property:

Let $E$ be an elliptic curve over $\Q$ given by a Weierstrass equation $y^2=f(x)$ such that $f(0)\ne 0$ and denote by $x:E\to\Pro^1$ the $x$-coordinate map. Let $\Gamma$ be a finitely generated multiplicative subgroup of $\Q^{\times}$ and let $\rho$ be its rank. Then
$$
\#\left(x(E(\Q))\cap \Gamma \right) \le \kappa\cdot 2^\rho.
$$
\end{theorem}

A special case is that of geometric progressions. The study of \emph{consecutive} terms of a geometric progression appearing as $x$-coordinates of rational points of an elliptic curve has captured some attention and we refer the reader to \cite{BremnerUlas, CissMoody, GFPpatterns, HMS} and the references therein. As a corollary of the previous result, we note that uniform boundedness can be expected even if the terms of the geometric progression are not consecutive.

\begin{corollary} If the ranks of elliptic curves over $\Q$ are uniformly bounded, then there is a uniform bound $B$ such that the following holds:

Let $E$ be an elliptic curve over $\Q$ given by a Weierstrass equation $y^2=f(x)$ such that $f(0)\ne 0$ and denote by $x:E\to\Pro^1$ the $x$-coordinate map. If $a, ab, ab^2,...$ is a geometric progression in $\Q^{\times}$ with $b\ne\pm 1$, then there are at most $B$ rational points of $E$ with $x$-coordinate in this geometric progression.
\end{corollary}

Indeed, one takes $\Gamma = \langle a,b\rangle$; note that at most two rational points have the same $x$-coordinate. The proof of Theorem \ref{ThmMult} is similar to that of Theorem \ref{ThmUniform} and is given in Section \ref{SecMult}.

%%%%%%
%%%%%%
%%%%%%

\subsection{Rank-dependent bound for multiplicative groups}

Given the recent developments on the height-uniform Mordell Conjecture (see Section \ref{SecMordell}) and, more generally, the height-uniform Mordell--Lang Conjecture \cite{GGK} it is reasonable to expect that a similar result for curves in semiabelian varieties is within reach (see also the comments in Section 1.3 of \cite{GGK} concerning semiabelian varieties). Let us state here the simplest case concerning a semiabelian variety that is not a torus or an abelian variety:

\begin{conjecture}[Height-uniform Mordell--Lang for a semiabelian case]\label{ConjSemiabelian} Let $d\ge 1$ be an integer. There is a constant $c(d)>1$ depending only on $d$ such that the following holds:

Consider the semiabelian surface $A=E\times \G_m$ where $E$ is an elliptic curve over $\Q$. Let $C\subseteq A$ be an irreducible curve defined over $\Q$ that is generically finite of degree $\le d$ over $E$ and over $\G_m$ under the coordinate projections $A\to E$ and $A\to\G_m$. Let $G\subseteq A(\Q)$ be a finitely generated group of rank $r$. Then
$$
\#(C\cap G) \le c(d)^{1+r}.
$$
\end{conjecture}

This statement is closely related to our discussion.
\begin{theorem} Assume Conjecture \ref{ConjSemiabelian}. There is a constant $K>1$ such that the following holds:

Let $E$ be an elliptic curve over $\Q$ of rank $r$ given by a Weierstrass equation $y^2=f(x)$ and denote by $x:E\to\Pro^1$ the $x$-coordinate map. Let $\Gamma$ be a finitely generated multiplicative subgroup of $\Q^{\times}$ and let $\rho$ be its rank. Then
$$
\#\left(x(E(\Q))\cap \Gamma \right) \le K^{\rho +r +1}.
$$
\end{theorem}
In fact, one uses Conjecture \ref{ConjSemiabelian} by choosing  $C\subseteq E\times \G_m$ as the graph of the $x$-coordinate rational map $x:E\dasharrow \G_m$. Then one takes $d=2$ and $G=E(\Q)\times \Gamma$. Of course there is nothing special about the $x$-coordinate map and a similar result (with other choices of $d$) holds for any non-constant map $E\to \Pro^1$.

For geometric progressions one gets a rank-dependent bound for the total number of terms coming from rational points of the elliptic curve, not only consecutive terms (again, at most two rational points have the same $x$-coordinate). 
\begin{corollary}[Rank-dependent bound for geometric progressions] Assume Conjecture \ref{ConjSemiabelian}. There is a constant $K>1$ such that the following holds:

Let $E$ be an elliptic curve over $\Q$ of rank $r$ given by a Weierstrass equation $y^2=f(x)$ and denote by $x:E\to\Pro^1$ the $x$-coordinate map. If $a, ab, ab^2,...$ is a geometric progression in $\Q^{\times}$ with $b\ne\pm 1$, then there are at most $K^{r+1}$ rational points of $E$ with $x$-coordinate in this geometric progression.
\end{corollary}
%%
%%

%%%%%%
%%%%%%
%%%%%%

\subsection{About the proof of Bremner's conjecture}\label{SecProofHistory} Bremner's conjecture was proved some years ago. It is an immediate consequence of our previous work \cite{GFP} in the following strong form:

\begin{theorem}[Strong form of Bremner's conjecture]\label{ThmBC} There is an absolute constant $C>1$ such that if $E$ is an elliptic curve over $\Q$ with rank $r$, then all arithmetic progressions on $E$ have length bounded by $C^{r+1}$.\end{theorem}

More precisely, in 2019 we proved this result with $C$ replaced by a quantity $C_0(j_E)$ that only depends on the $j$-invariant $j_E$ of $E$, which was strong enough to prove Bremner's conjecture over twist families, see \cite{GFP}. The core of the proof is a technical Nevanlinna-theoretical argument that reduced the problem to a suitable version of the Mordell--Lang conjecture for surfaces contained in abelian varieties. At the time, only R\'emond's quantitative version of Mordell--Lang was available \cite{Remond1, Remond2}, whose constants depended on the Faltings height of the abelian variety. But the situation improved in 2021 when Gao--Ge--K\"uhne \cite{GGK} removed the dependence on the height and, in our setting, this resulted in removing the dependence on $j_E$ (for the interested reader: this change must be made in the first paragraph after the proof of Lemma 6.5 in \cite{GFP} when one introduces the constant $c(E^n,\Lcal_n)$; everything else in the proof of Theorem 6.1 is the same and the statement gets upgraded accordingly). This is mentioned, for instance, in \cite{Choi}. 

It should be noted that this state of affairs is not unique to our theorem: it is well-known to experts that most ---if not all!--- applications of R\'emond's bound became height-uniform thanks to the Gao--Ge--K\"uhne theorem. In fact, after our proof of Bremner's conjecture in \cite{GFP}, the idea of studying additive patterns in the $x$-coordinates of rational points of elliptic curves via the Gao--Ge--K\"uhne theorem was further developed by Caro and the first author of this note in \cite{CGF}.

 We remark that another proof of Theorem \ref{ThmBC} was recently claimed in the preprint  \cite{Choi}, although it assumes a certain height conjecture of Lang, so at present, that proof currently applies only to certain restricted families of elliptic curves. Yet another new (unconditional) proof of Bremner's conjecture in the form of Theorem \ref{ThmBC} was very recently claimed in the preprint \cite{HMS} using different techniques. Both arguments are independent of our original proof of Theorem \ref{ThmBC} and they have their own features.
 
Finally, we mention two technical points for experts: our Theorem \ref{ThmBC} holds over any number field (not just $\Q$), and one can replace the $x$-coordinate map in the definition of arithmetic progression by other rational functions (for instance, $y$-coordinates) leading to the same kind of result. This is in fact the way we run the argument in Theorem 6.1 of \cite{GFP}.

%%%%%%%%%%%%%%%%%%%%%%%%%%%%%%%%%%%%%%
%%%%%%%%%%%%%%%%%%%%%%%%%%%%%%%%%%%%%%
%%%%%%%%%%%%%%%%%%%%%%%%%%%%%%%%%%%%%%
%%%%%%%%%%%%%%%%%%%%%%%%%%%%%%%%%%%%%%
%%%%%%%%%%%%%%%%%%%%%%%%%%%%%%%%%%%%%%
%%%%%%%%%%%%%%%%%%%%%%%%%%%%%%%%%%%%%%

\section{Height-uniform Mordell} \label{SecMordell}

In 1922, Mordell \cite{Mordell} proposed his celebrated conjecture on rational points of curves:

\begin{conjecture}[Mordell] Let $k$ be a number field and $X$ a smooth projective curve of genus $g\ge 2$ defined over $k$. Then the set of rational points $X(k)$ is finite. 
\end{conjecture}

This conjecture was proved by Faltings in his spectacular work \cite{Faltings1}. A second proof with completely different ideas was produced by Vojta in \cite{Vojta}. While Vojta's initial argument was highly sophisticated and used Arakelov geometry, Bombieri \cite{Bombieri} translated it into classical diophantine approximation terms that were later further developed by Faltings to solve the Mordell--Lang conjecture \cite{Faltings2,Faltings3}.  

At the core of Vojta's proof there is a gap phenomenon first discovered by Mumford \cite{Mumford}. From this gap phenomenon it was possible to extract a bound for the \emph{number} of rational points. In a remarkably explicit work, R\'emond \cite{Remond1,Remond2} succeeded in finding such bounds in the more general context of the Mordell--Lang conjecture. For curves (i.e. in the context of Mordell's conjecture), R\'emond's bound crucially depended on
\begin{itemize}
\item The genus of the curve
\item The Mordell--Weil rank of the Jacobian, and
\item The Faltings height of the Jacobian.
\end{itemize}

Conjecturally, there was room for improvement. For instance, Caporaso--Harris--Mazur \cite{CHM} showed that the Bombieri--Lang conjecture \emph{in all dimensions} implies that for any integer $g\ge 2$ and number field $k$, there is a bound $B(g,k)$ depending only on $g$ and $k$ such that every smooth projective curve $X$ of genus $g$ defined over $k$ satisfies 
$$
\# X(k)\le B(g,k). 
$$
A somewhat milder problem was asked by Mazur \cite{Mazur}:  Is there a bound for $\# X(k)$ which is independent of the Faltings height of the Jacobian (although it could still depend on the Mordell--Weil rank)?

In 2020, Dimitrov--Gao--Habegger \cite{DGH} finally answered Mazur's question by proving the following remarkable height-uniform upper bound:

\begin{theorem}[Height-uniform Mordell]\label{ThmDGH} Let $g\ge 2$ and $d\ge 1$ be integers. There is a constant $c=c(g,d)$ depending only on $g$ and $d$ such that if $X$ is a smooth projective curve of genus $g$ defined over a number field $k$ of degree $d$ over $\Q$, then 
$$
\# X(k)\le c^{1+\rho}
$$ 
where $\rho$ is the Mordell--Weil rank of the Jacobian of $X$ over $k$.
\end{theorem}

After the Dimitrov--Gao--Habegger paper, there have been several extensions and, most recently, a completely explicit height-uniform upper bound has been obtained in \cite{YYZ}.

%%%%%%%%%%%%%%%%%%%%%%%%%%%%%%%%%%%%%%
%%%%%%%%%%%%%%%%%%%%%%%%%%%%%%%%%%%%%%
%%%%%%%%%%%%%%%%%%%%%%%%%%%%%%%%%%%%%%
%%%%%%%%%%%%%%%%%%%%%%%%%%%%%%%%%%%%%%
%%%%%%%%%%%%%%%%%%%%%%%%%%%%%%%%%%%%%%
%%%%%%%%%%%%%%%%%%%%%%%%%%%%%%%%%%%%%%

\section{Bremner's uniformity question} \label{SecProof}

\begin{proof}[Proof of Theorem \ref{ThmUniform}] Let $E$ be an elliptic curve over $\Q$ with Weierstrass equation $y^2=f(x)$ where $f\in \Q[x]$ is a monic cubic polynomial. Consider an arithmetic progression on $E$ of length $M\ge 4$. Then $E$ contains an arithmetic progression of length $N=M-3$ whose first term is not a $2$-torsion point of $E$. Let 
$$
b, a+b, 2a+b, ..., (N-1)a+b
$$ 
be the $x$-coordinates of this arithmetic progression and note that $f(b)\ne 0$. Consider the equation
$$
s^2 = f(at^2+b).
$$
The hexic polynomial $f(at^2+b)$ has no repeated roots because $f$ is separable, $a\ne 0$, and $f(b)\ne 0$. So, the previous equation defines a hyperelliptic curve $X$ of genus $2$. This curve comes with some rational points: at least those with $t$ coordinates
$$
t= -n,...,-1, 0, 1, 2, ..., n
$$
where $n=\lfloor\sqrt{N-1}\rfloor$. This produces at least $2n+1$ different rational points in $X$, that is,
$$
2n+1 \le \# X(\Q).
$$

The substitution 
$$
y=s, \quad x=at^2+b
$$
defines a non-constant map $\pi: X\to E$.

The Jacobian $J$ of $X$ is an abelian surface because $X$ has genus $2$, and the above map exhibits $E$ as an isogeny factor of $J$ over $\Q$. Therefore $J$ splits as $J\sim E\times E'$ up to isogeny over $\Q$ where $E'$ is another elliptic curve (which in general is not isogenous to $E$). 

If the ranks of elliptic curves over $\Q$ are bounded by a constant $R$, then 
$$
\rk J(\Q) = \rk E(\Q) + \rk E'(\Q) \le 2R.
$$

The Dimitrov--Gao--Habegger theorem (Theorem \ref{ThmDGH}) provides an absolute constant $c>1$ (in particular, independent of $E$ and $X$) such that
$$
\# X(\Q) \le c^{1+\rk J(\Q)} \le c^{1+2R}.
$$
This gives $2\lfloor \sqrt{M-4} \rfloor +1 =2n+1 \le c^{1+2R}$, which shows that $M$ is uniformly bounded (assuming the existence of the uniform rank bound $R$).
\end{proof}

%%%%%%%%%%%%%%%%%%%%%%%%%%%%%%%%%%%%%%
%%%%%%%%%%%%%%%%%%%%%%%%%%%%%%%%%%%%%%
%%%%%%%%%%%%%%%%%%%%%%%%%%%%%%%%%%%%%%
%%%%%%%%%%%%%%%%%%%%%%%%%%%%%%%%%%%%%%
%%%%%%%%%%%%%%%%%%%%%%%%%%%%%%%%%%%%%%
%%%%%%%%%%%%%%%%%%%%%%%%%%%%%%%%%%%%%%

\section{Finitely generated multiplicative groups} \label{SecMult}

\begin{proof}[Proof of Theorem \ref{ThmMult}] This is very similar to the argument in the previous section. We assume that the ranks of elliptic curves over $\Q$ are bounded by some $R$.

Let $E$ and $f(x)$ be as in the statement. For every $a\in\Q^\times$ the equation $y^2=f(at^2)$ defines a genus $2$ hyperelliptic curve $X_a$ over $\Q$ because $f$ is separable and $f(0)\ne 0$. Its Jacobian $J_a$ splits over $\Q$ as $J_a\sim E\times E_a$ for some elliptic curve $E_a$ depending on the choice of $a$. This splitting comes from the map $X_a\to E$ defined by $(t,y)\mapsto (at^2,y)$.

Theorem \ref{ThmDGH} gives an absolute constant $c>1$ (in particular, independent of $E$ and $a$) such that
$$
\# X_a(\Q) \le c^{1+\rk J_a(\Q)} \le c^{1+2R}.
$$

Let $a$ vary over a set of representatives of $\Gamma/\Gamma^2$; this is at most $2^{\rho+1}$ choices of $a\in \Q^\times$. Each $\gamma\in x(E(\Q))\cap \Gamma$ gives a rational point in some $X_a$; namely, choose $a$ as the representative of $[\gamma]\in\Gamma/\Gamma^2$. The result follows with $\kappa = 2c^{1+2R}$.
\end{proof}

%%%%%%%%%%%%%%%%%%%%%%%%%%%%%%%%%%%%%%
%%%%%%%%%%%%%%%%%%%%%%%%%%%%%%%%%%%%%%
%%%%%%%%%%%%%%%%%%%%%%%%%%%%%%%%%%%%%%
%%%%%%%%%%%%%%%%%%%%%%%%%%%%%%%%%%%%%%
%%%%%%%%%%%%%%%%%%%%%%%%%%%%%%%%%%%%%%
%%%%%%%%%%%%%%%%%%%%%%%%%%%%%%%%%%%%%%
%%%%%%%%%%%%%%%%%%%%%%%%%%%%%%%%%%%%%%
%%%%%%%%%%%%%%%%%%%%%%%%%%%%%%%%%%%%%%
%%%%%%%%%%%%%%%%%%%%%%%%%%%%%%%%%%%%%%
%%%%%%%%%%%%%%%%%%%%%%%%%%%%%%%%%%%%%%
%%%%%%%%%%%%%%%%%%%%%%%%%%%%%%%%%%%%%%
%%%%%%%%%%%%%%%%%%%%%%%%%%%%%%%%%%%%%%

\section{Acknowledgments}

N.G.-F. was supported by ANID Fondecyt Regular grant 1251300 from Chile. H.P. was supported by ANID Fondecyt Regular grant 1230507 from Chile. We thank Yuri Bilu for a comment on a previous version that helped us to pay more attention to the split Jacobian point of view.

%%%%%%%%%%%%%%%%%%%%%%%%%%%%%%%%%%%%%%
%%%%%%%%%%%%%%%%%%%%%%%%%%%%%%%%%%%%%%
%%%%%%%%%%%%%%%%%%%%%%%%%%%%%%%%%%%%%%


\begin{thebibliography}{9}         

\bibitem{Bombieri} E. Bombieri, \emph{The Mordell conjecture revisited}. Annali della Scuola Normale Superiore di Pisa-Classe di Scienze, 1990, vol. 17, no 4, p. 615-640.

\bibitem{Bremner} A. Bremner, \emph{On arithmetic progressions on elliptic curves}. Experimental Mathematics, 1999, vol. 8, no 4, p. 409-413.

\bibitem{BremnerUlas} A. Bremner, M. Ulas, \emph{Rational points in geometric progressions on certain hyperelliptic curves}. Publ. Math. Debrecen 82.3--4 (2013): 669-683.

\bibitem{BST} A. Bremner, J. Silverman, N. Tzanakis, \emph{Integral points in arithmetic progression on $y^2=x(x^2-n^2)$}. Journal of Number Theory, (2000) 80(2), 187-208.

\bibitem{CHM} L. Caporaso, J. Harris, B. Mazur, \emph{Uniformity of rational points}. J. Amer. Math. Soc. 10 (1997), no. 1, 1-35.

\bibitem{CGF} J. Caro, N. Garcia-Fritz, \emph{Linear $x$-coordinate relations of triples on elliptic curves}. Journal of Number Theory 271 (2025): 109-121.

\bibitem{Choi} S. Choi, \emph{Additive rigidity for $x$-coordinates of rational points on elliptic curves}. Preprint (2026) (formerly: Elliptic curves and rational points in arithmetic progression, 2025).  arXiv:2510.03828

\bibitem{CissMoody} A. Ciss, D. Moody, \emph{Geometric progressions on elliptic curves}. Glasnik matematicki 52.1 (2017): 1-10.

\bibitem{DGH} V. Dimitrov, Z. Gao, P. Habegger, \emph{Uniformity in Mordell-Lang for curves}. Ann. of Math. (2) 194 (2021), no. 1, 237-298. 

\bibitem{Faltings1} G. Faltings, \emph{Endlichkeitss\"atze f\"ur abelsche Variet\"aten \"uber Zahlk\"orpern}. (German) [Finiteness theorems for abelian varieties over number fields] Invent. Math. 73 (1983), no. 3, 349-366.


\bibitem{Faltings2} G. Faltings, \emph{Diophantine approximation on abelian varieties}. Annals of Mathematics, 133 (1991), 549-576.

\bibitem{Faltings3} G. Faltings, \emph{The general case of S. Lang's conjecture}. Barsotti Symposium in Algebraic Geometry (Abano Terme, 1991). Perspect. Math. 15. Academic Press. San Diego. 1994. p. 175-182

\bibitem{GF} N. Garcia-Fritz, \emph{Quadratic sequences of powers and Mohanty's conjecture}. International Journal of Number Theory 14.02 (2018), 479-507.

\bibitem{GFP} N. Garcia-Fritz, H. Pasten, \emph{Elliptic curves with long arithmetic progressions have large rank}. Int. Math. Res. Not. IMRN 2021, no. 10, 7394-7432.

\bibitem{GFPpatterns} N. Garcia-Fritz, H. Pasten, \emph{Patterns on elliptic curves beyond Bremner's conjecture}. Preprint (2026) arXiv:2605.14962

\bibitem{GGK} Z. Gao, T. Ge, L. K\"uhne, \emph{The Uniform Mordell-Lang Conjecture}. (2021) to appear in Publ. math. IHES.

\bibitem{HMS} J. Harrison, A. Mudgal, H. Schmidt, \emph{Uniform sum-product phenomenon for algebraic groups and Bremner's conjecture}. Preprint (2026) arXiv:2603.06483

\bibitem{Mazur} B. Mazur, \emph{Abelian varieties and the Mordell-Lang conjecture}. In: Model Theory, Algebra, and Geometry, Math. Sci. Res. Inst. Publ. 39, Cambridge Univ. Press, Cambridge, 2000, pp. 199-227.

\bibitem{Mordell} L. Mordell, \emph{On the rational solutions of the indeterminate equations of the third and fourth degrees}. Proc. Cambridge Philos. Soc. 21 (1922/23), 179-192.

\bibitem{Mumford} D. Mumford, \emph{A remark on Mordell's conjecture}. 
Amer. J. Math. 87 (1965), 1007-1016. 

\bibitem{PPVW} J. Park, B. Poonen, J. Voight, M. Wood, \emph{A heuristic for boundedness of ranks of elliptic curves}. J. Eur. Math. Soc. (JEMS) 21 (2019), no. 9, 2859-2903. 

\bibitem{Remond1} G. R\'emond, \emph{D\'ecompte dans une conjecture de Lang}. Inventiones Mathematicae, (2000) 142 (3), 513-545.

\bibitem{Remond2} G. R\'emond, \emph{Sur les sous-vari\'et\'es des tores}. Compositio Mathematica 134.3 (2002) 337-366.

\bibitem{Vojta} P. Vojta, \emph{Siegel's theorem in the compact case}. Annals of Mathematics, 1991, p. 509-548.

\bibitem{YYZ} J. Yu, X. Yuan, S. Zhou, \emph{Quantitativity on the number of rational points in the Mordell conjecture}. Preprint (2026) arXiv:2602.01820

\end{thebibliography}
\end{document}